\documentclass[a4paper,10pt,twoside]{amsart}
\usepackage{amsthm}
\usepackage{amsmath}
\usepackage{amssymb}
\usepackage{amsfonts}
\usepackage{amscd}
\usepackage[english]{babel}
\usepackage{stmaryrd}
\usepackage{enumerate}

\newtheorem{thm}{Theorem}

\newtheorem{prop}{Proposition}
\newtheorem*{rem}{Remark}

\newtheorem*{defi}{Definition}

\begin{document}

    \title{What about $A,B$ if $AB-BA$ and $A$ commute.}
    \author{Gerald BOURGEOIS}
    
    \date{16-03-2011}
    \address{G\'erald Bourgeois, GAATI, Universit\'e de la polyn\'esie fran\c caise, BP 6570, 98702 FAA'A, Tahiti, Polyn\'esie Fran\c caise.}
    \email{gerald.bourgeois@upf.pf}
        

\begin{abstract}
Let $A,B$ be complex $n\times{n}$ complex matrices such that $AB-BA$ and $A$ commute. We show that, if $n=2$ then $A,B$ are simultaneously triangularizable and if $n\geq{3}$ then there exists such a couple $A,B$ such that the pair $(A,B)$ has not property L of Motzkin-Taussky and such that $B$ and $C$ are not simultaneously triangularizable.
\end{abstract}

\maketitle
\textbf{Notations.}
$i)$ If $U$ is a square matrix, then $\sigma(U)$ denotes the spectrum of $U$.\\
$ii)$ Let $A,B$ be complex $n\times{n}$ complex matrices. If there exists an invertible matrix $P$ such that $P^{-1}AP$ and $P^{-1}BP$ are upper triangular then we say that $A$ and $B$ are $ST$.\\
$iii)$ Denote by $I_n$ and $0_n$ the identity matrix and the zero matrix of dimension $n$.
\begin{defi}
(See \cite{3}). A pair $(A,B)$ of complex $n\times{n}$ matrices is said to have property L if for a special ordering of the eigenvalues $(\lambda_i)_{i\leq{n}},(\mu_i)_{i\leq{n}}$ of $A,B$, the eigenvalues of $xA+yB$ are $(x\lambda_i+y\mu_i)_{i\leq{n}}$, for all values of the complex numbers $x,y$. 
\end{defi}
\begin{rem}
If $A,B$ are $ST$ then $(A,B)$ has property L, but the converse is false (see \cite{3}).
\end{rem}
We deal with the couples $(A,B)$ such that $AB-BA$ and $A$ commute. If $(A,B)$ is such a couple , then for every complex numbers $\lambda,\mu$, $(A+\lambda{I}_n,B+\mu{I}_n)$ is another one. Then we may assume that $A$ and $B$ are invertible matrices, or on the contrary, that they are singular. In the sequel, we put $C=AB-BA$. \\
Several well-known results are gathered in the following Proposition. 
\begin{prop}
Let $A,B$ be complex $n\times{n}$ matrices. We assume that $C$ and $A$ commute. Then $C$ is a nilpotent matrix and the pair $(B,C)$ has property L of Motzkin-Taussky.  Moreover, if $A,B$ are invertible, then $A^{-1}B^{-1}C,B^{-1}A^{-1}C$ and $B^{-1}C$ are nilpotent matrices.
\end{prop}
\begin {proof}
$C$ is nilpotent by vertue of \cite{1}. According to \cite{2}, one has, for every real $t$, $e^{tA}Be^{-tA}=B+tC$ and therefore $\sigma(B+tC)=\sigma(B)$. Reasoning by a continuity argument, we can conclude that the pair $(B,C)$ has property L.  \\
Now we assume that $A,B$ are invertible. One has $A^{-1}CB^{-1}=CA^{-1}B^{-1}=ABA^{-1}B^{-1}-I_n$. In \cite [Theorem 2]{4}, it is shown that  $ABA^{-1}B^{-1}-I_n$ is a nilpotent matrix. Since $\sigma(A^{-1}B^{-1}C)=
\sigma(CA^{-1}B^{-1})=\{0\}$ and $\sigma(B^{-1}A^{-1}C)=\sigma(A^{-1}CB^{-1})=\{0\}$, we conclude that 
$A^{-1}B^{-1}C$ and $B^{-1}A^{-1}C$ are also nilpotent matrices. Finally the fact that $CB^{-1}$ is nilpotent (or equivalently $B^{-1}C$ is nilpotent) is also proven in \cite{4} (see the proof of theorem 1).
\end{proof}
There are strong relations on the one hand between $A$ and $C$ and on the other hand between $B$ and $C$.
We may wonder whether $A$ and $B$ are simultaneously triangularizable or, at least, the pair $(A,B)$ has property L. We have a positive answer in the following case.
\begin{defi}
A complex matrix $A$ is said to be \emph{non-derogatory} if for all $\lambda\in\sigma(A)$, the number of Jordan blocks of $A$ associated with $\lambda$ is $1$.
\end{defi}
\begin{prop}
If $A$ is a non-derogatory matrix and if $AC=CA$, then $A$ and $B$ are $ST$.
\end{prop}
\begin{proof}
Necessarily, $C$ is a polynomial in $A$. According to \cite [Theorem 1] {5}, $A$ and $B$ are $ST$.
\end{proof}
\begin{rem}
$i)$ Note that the set of derogatory matrices is included in the set $NS$ of non-separable matrices (they have at least one multiple eigenvalue). $NS$ is an algebraic variety in $\mathcal{M}_n(\mathbb{C})$ of codimension $1$ and therefore is a null set in the sense of Lebesgue measure (see \cite{6} for an outline of the proof).\\   
$ii)$ If we fix the matrix $A$, then the equation $A(AB-BA)=(AB-BA)A$ is linear in the unknown $B$. More precisely $B\in{k}er(\phi)$ where $\phi:X\rightarrow{A}^2X+XA^2-2AXA$. Therefore $\phi=A^2\otimes{I}+I\otimes{(A^T)^2}-2A\otimes{A}^T=\psi^2$ where $\psi=A\otimes{I}_n-I_n\otimes{A}^T$. Thus, if $\sigma(A)=(\lambda_i)_i$, then $\sigma(\psi)=(\lambda_i-\lambda_j)_{i,j}$ and $\sigma(\phi)=((\lambda_i-\lambda_j)^2)_{i,j}$. We deduce that the expression $i(A)=\dfrac{dim(ker(\psi^2))-dim(ker(\psi)}{n^2}$ is linked to the existence of $B$ such that $AB-BA$ and $A$ commute and such that $A,B$ are not $ST$.
\end{rem}
 Now we prove our main result.
\begin{prop}
i) If $n=2$ and $CA=AC$, then $A$ and $B$ are  $ST$.\\ 
ii) If $n\geq{3}$ then there exists a couple $A,B$ such that $AB-BA$ and $A$ commute, satisfying\\
$\bullet$ the pair $(A,B)$ has not property L.\\
$\bullet$ $B$ and $C$ are not $ST$.
\end{prop}
\begin{proof}
$i)$ According to a previous remark, we may assume that $A$ is derogatory, that is a scalar matrix, and we conclude immediately. \\
$ii)$ It is sufficient to find such a counterexample $(A_0,B_0)$ when $n=3$. Indeed, if $n>3$, then consider the couple $(A_0\bigoplus{0}_{n-3},B_0\bigoplus{0}_{n-3})$.\\
Now we suppose $n=3$ and $A_0$ is chosen as a derogatory matrix, for instance $A_0=\begin{pmatrix}0&1&0\\0&0&0\\0&0&0\end{pmatrix}$. Here $\psi$ is nilpotent, $dim(ker(\psi))=5$, $dim(ker(\psi^2))=8$ and $i(A_0)=\dfrac{1}{3}$. We search associated matrices $B$. Amongst numerous solutions, we choose this one\\ $$B=\begin{pmatrix}0&0&0\\0&0&1\\1&0&0\end{pmatrix}.$$
$\bullet$ $(A_0,B)$ has not property L because $\sigma(A_0)=\{0\}$ and for every couple of complex numbers $(t,x)$, $\chi_{tA_0+B}(x)=x^3-t$.\\
$\bullet$ We observe that $Trace(B^2C^2)=-1$, that implies that $B$ and $C$ are not $ST$.
\end{proof}
To show that two complex matrices are $ST$, the McCoy Theorem (see \cite{8}) contains no finite verification procedure. The following test admits a finite one (see \cite[Theorem 6]{7}).
 \begin{thm} \label{test} Two $n\times{n}$ complex matrices $A$ and $B$ are $ST$ if and only if\\ 
for every $k\in[[1,n^2-1]]$, each matrix in the form $U_1\cdots{U}_k(AB-BA)$ (where, for every $i$, $U_i$ is $A$ or $B$) has a zero trace.
\end{thm}
\begin{prop}
If $n\geq{4}$, then there exist derogatory matrices $A_1$ such that $A_1$ and each associated matrix $B$ are $ST$.
\end{prop}
\begin{proof}
 We take $n=4$ and $A_1=diag(J_2,J_2)$ where $J_2$ is the Jordan nilpotent block of dimension $2$. Here $dim(ker(\psi))=8$, $dim(ker(\psi^2))=12$ and $i(A_1)=\dfrac{1}{4}<i(A_0)$. The associated matrices $B$ are in the form $B=\begin{pmatrix}*&*&*&*\\0&*&0&*\\**&*&*&*\\0&*&0&*\end{pmatrix}$ where each $*$ represents an arbitrary complex entry. 
Using Theorem \ref{test}, we verify (with Maple software) that the $65534$ considered matrices have a zero trace. Thus $A_1$ and $B$ are $ST$.
\end{proof}

\textbf{Acknowledgments}\\
The author thanks professor R. Horn for bringing this problem to his attention.

\bibliographystyle{plain}

\begin{thebibliography}{99}
\bibitem[1]{1} D.C. Kleinecke. On operator commutators. Proc. Amer.Math. Soc. vol 8 (1957), p.535-536.
\bibitem[2]{2} J.E. Campbell. On a law of combination of operators bearing on the theory of continuous transformation groups. Proc. London Math. Soc. vol 28 (1897) p. 381-390.
\bibitem[3]{3} T.S. Motzkin, O. Taussky. Pairs of matrices with property L. II . Transactions of the AMS. Vol 80. N° 2 (1955). p. 387-401.
\bibitem[4]{4} C.R. Putnam, A. Wintner. On the spectra of group commutators. Proc. Amer. Math. Soc. Vol 9 (1958). p. 360-362.
\bibitem[5]{5} G. Bourgeois. How to solve the matrix equation $XA-AX=f(X)$. Linear Algebra and Applications. Vol. 434, issue 3 (2011). p. 657-668. 
\bibitem[6]{6} P. Neumann, C. Praeger. Cyclic matrices over finite fields. J. London. Math. Soc. (2) 52 (1995). p.263-284.
\bibitem[7]{7}Y. Alp'in, N. Koreshkov. On the simultaneous triangulability of matrices. Mathematical notes, Vol. 68, No. 5, 2000.
\bibitem[8]{8} R. Horn, C. Johnson. Matrix analysis. Cambridge university press, Cambridge, 1985.
\end{thebibliography}

\end{document}